\newif\ifpictures
\picturestrue

\documentclass[12pt]{amsart} 
\usepackage{amssymb}
\usepackage{amsmath}
\usepackage{amscd}
\usepackage[pdftex]{graphicx}
\usepackage{hyperref}
\usepackage{bbm} 
\usepackage[textsize=tiny]{todonotes}

\numberwithin{equation}{section}
\newtheorem{thm}{Theorem}
\newtheorem{prop}[thm]{Proposition}
\newtheorem{lemma}[thm]{Lemma}
\newtheorem{cor}[thm]{Corollary}

\theoremstyle{definition}

\newtheorem{example}[thm]{Example}
\newtheorem{remark1}[thm]{Remark}
\newtheorem{openproblem1}[thm]{Open problem}
\newtheorem{definition}[thm]{Definition}

\newenvironment{ex}{\begin{example}\rm}{\hfill$\Box$\end{example}}

\numberwithin{thm}{section}

\newcounter{FNC}[page]
\def\newfootnote#1{{\addtocounter{FNC}{2}$^\fnsymbol{FNC}$%
		\let\thefootnote\relax\footnotetext{$^\fnsymbol{FNC}$#1}}}

\newcommand{\C}{\mathbb{C}}
\newcommand{\N}{\mathbb{N}}

\newcommand{\R}{\mathbb{R}}

\newcommand{\Z}{\mathbb{Z}}

\newcommand{\cC}{\mathcal{C}}

\newcommand{\one}{\mathbf{1}} 

\DeclareMathOperator{\supp}{supp}

\DeclareMathOperator{\diag}{diag}

\DeclareMathOperator{\tr}{tr}

\DeclareMathOperator{\argmax}{argmax}

\DeclareMathOperator{\im}{im}

\DeclareMathOperator{\rank}{rank}

\newcommand{\sym}{\mathcal{S}}

\makeatletter
\def\@settitle{\begin{center}%
		\baselineskip14\p@\relax
		\bf\Large
		\@title
	\end{center}%
}
\makeatother

\title[Nash equilibria in semidefinite games and Lemke-Howson paths]{Nash 
	equilibria in semidefinite games and Lemke-Howson paths}

\author[C.~Ickstadt]{Constantin Ickstadt}

\author[T.~Theobald]{Thorsten Theobald}

\address{Constantin Ickstadt, Thorsten Theobald:
        Goethe-Universit\"at, FB 12 -- Institut f\"ur Mathematik,
        Postfach 11 19 32, 60054 Frankfurt am Main, Germany}

\email{ickstadt@math.uni-frankfurt.de,
        theobald@math.uni-frankfurt.de}

\author[E.~Tsigaridas]{Elias Tsigaridas}

\address{Elias Tsigaridas: Sorbonne Universit\'e, Paris University, CNRS and Inria Paris.
        IMJ-PRG,  4 place Jussieu,
        75252 Paris Cedex 05, France}

\email{elias.tsigaridas@inria.fr}

\author[A.~Varvitsiotis]{Antonios Varvitsiotis}
\address{Antonios Varvitsiotis: Singapore University of Technology and Design (ESD Pillar),
        8 Somapah Road 487372, Singapore}

\email{avarvits@gmail.com}

\date{\today}

\begin{document}

\begin{abstract}
 We consider an algorithmic framework for two-player non-zero-sum semidefinite games, where each player's strategy is a positive semidefinite matrix with trace one.
We formulate the computation of Nash equilibria in such games as semidefinite complementarity problems, and develop symbolic–numeric techniques to trace generalized Lemke–Howson paths. These paths generalize the piecewise affine-linear trajectories of the classical Lemke–Howson algorithm for bimatrix games, replacing them with nonlinear curve branches governed by eigenvalue complementarity conditions.

A key feature of our framework is the introduction of event points, which correspond to curve singularities. We analyze the local behavior near these points using Puiseux series expansions.
We prove the smoothness of the curve branches under suitable non-degeneracy conditions and establish connections between our approach and both the classical combinatorial and homotopy-theoretic interpretations of the Lemke–Howson algorithm.
\end{abstract}

\maketitle

\section{Introduction}

In the classical model of a bimatrix game,  Nash equilibria can be characterized  as
the solutions of a linear complementarity problem with a product
structure \cite{cps-1992}. The combinatorial description of these
linear complementarity problems
gives rise to an algorithm
for computing a Nash equilibrium, known as
\emph{Lemke-Howson algorithm}~\cite{lemke-howson-1964}. 
This approach provides a
piecewise linear path in an extended Cartesian product of strategy spaces
and the algorithmic key idea is to follow this path until we reach a Nash equilibrium.
The path, known as
\emph{Lemke-Howson path} and an underlying combinatorial graph are cornerstones
for both structural and computational results in bimatrix game theory
\cite{shapley-1974,von-stengel-handbook}.
Conceptually, 
Rosenm\"uller \cite{rosenmueller-1971} extended 
the Lemke-Howson algorithm for finite $N$-person games, 
where the affine-linear curve branches  are replaced by nonlinear curve branches.

There has been significant recent interest in the study of \emph{quantum games}, where players can store, process, and exchange quantum information; see, e.g.,~\cite{guo-2008,watrous1,meyer-1999,ksb-2018,unitary}. A notable subclass  are \emph{semidefinite games}, involving  players whose strategies are represented by \emph{density matrices}, i.e., Hermitian (or symmetric) positive semidefinite matrices with trace equal to one, and whose utilities are multilinear functions in these strategies~\cite{itt-2024}.

For two-player zero-sum semidefinite games, Nash equilibria are guaranteed to exist and we can compute them efficiently using semidefinite programming (SDP)~\cite{itt-2024}. 
This framework has been further extended to \emph{zero-sum network semidefinite games}~\cite{ittv-2023}. Additionally, no-regret learning algorithms have been developed for computing equilibria in zero-sum semidefinite games~\cite{lpsv-2024,cig,quadratic}.

However, beyond the zero-sum case, the problem of computing or approximating Nash equilibria in general semidefinite games remains largely open. Notably, Bostanci and Watrous~\cite{watrous2} posed the specific question of whether a Lemke-Howson-style algorithm could be designed for computing or approximating Nash equilibria in the context of non-zero-sum semidefinite games.

We take on this challenge and we develop the theory of the structure of the generalized Lemke-Howson
paths for two-player semidefinite games. 
We rely on the theory of semidefinite games
to introduce naturally the generalized Lemke-Howson paths. 
In turn, these provide the basis for
symbolic-numeric algorithmic approaches to compute Nash equilibria in
non-zero-sum semidefinite games.

The first step of our study consists of providing a comprehensive eigenvalue view on the Nash
equilibria in semidefinite games. 
Then, by exploiting the formulation of the Nash equilibria
in semidefinite games as a semidefinite complementarity problem~\cite{ittv-2023},
we reveal the interplay of the eigenvalues of the underlying matrices for the
Nash equilibrium problem in Section~\ref{se:eigenvalue-view}. 
Consequently, using rank characterizations of the
strict complementarity conditions from semidefinite programming, we turn
the eigenvalue characterizations into a non-degeneracy notion for semidefinite games.
For the special case of diagonal semidefinite games, which correspond to bimatrix 
games, our approach specializes to the (classical) set of equivalent non-degeneracy notions, presented by von Stengel~\cite[Chapter~45]{von-stengel-handbook}.

These tools, based on eigenvalues, provide us with the means
to study the generalized Lemke-Howson paths.
We replace the combinatorial view of bimatrix games by
the interplay of eigenvalues in semidefinite games
and we replace the bilinear complementarity
conditions by bilinear polynomial equations
coming from a product of two positive semidefinite matrices (equal to zero).
In this way, we generalize the Lemke-Howson paths
by providing a distinguished finite set of points,
which we call \emph{event points},
where the combinatorics of the eigenvalue complementarity pairing changes, within an extended pair
of strategy spaces. The event points are interconnected by curve branches. 
While the curve branches are no longer affine-linear as in the case of bimatrix games, 
but inherently nonlinear, they still retain similar combinatorial properties as the
affine-linear pieces from the bimatrix setting.
We address some algebraic and numerical aspects
of computing, identifying, and tracing the non-linear curve branches,
and also how to pick the correct next branch
in an event point, see Section~\ref{se:tracing}. Locally at an event point, the
branches are described in terms of Puiseux series.

In settings related to numerical algorithms for tracing or approximating
a specific algebraic curve in a finite-dimensional space,
an essential question is whether the underlying algebraic curve is smooth.
For example, the central path in semidefinite programming, which is numerically
followed by interior point methods, is known to be a smooth curve,
see \cite{deklerk-2003}. For the Lemke-Howson paths, we show in Section~\ref{se:smooth}
that under the non-degeneracy condition, 
the curve branches between two event points are smooth.

One basis of our treatment is the classical combinatorial view on the 
Lemke-Howson algorithm. A second fruitful point of departure is
to view the Lemke-Howson
algorithm for bimatrix games as a homotopy algorithm; see Govindan and
Wilson \cite[Section 5.4]{govindan-wilson-2003}
and Herings and Peeters \cite{herings-peeters-2010}.
To prepare for the generalized semidefinite setting, it is beneficial to connect the
combinatorial view on the Lemke-Howson approach with the homotopy view. Therefore,
we review the Lemke-Howson algorithm for bimatrix games
in Section~\ref{se:lemke-howson-bimatrix} from the point of view
of a homotopy algorithm.

To ensure that the homotopy methods work generically for bimatrix games,
we use the Kohlberg-Mertens structure theorem~\cite{kohlberg-mertens-1986}.
This provides a topological description of the "graph of the
Nash equilibrium correspondence".
The Kohlberg-Mertens structure theorem is also valid in more general
semialgebraic contexts, see Bich and Fixary~\cite{bich-fixary-2024},
and in particular it holds for semidefinite games.

\section{Preliminaries\label{se:prelim}}

\subsection{Nash equilibria in semidefinite games}

Let $\mathcal{S}_n$ be the set of symmetric $n \times n$ matrices and
$\mathcal{S}_n^+ \subset \mathcal{S}_n$ be the subset of positive semidefinite matrices.
We consider semidefinite games between two players with the strategy spaces of real-valued density matrices $\mathcal{X} = \{ X \in \sym_m^+ \, : \, \tr(X) = 1\}$ and $\mathcal{Y} = \{Y \in \sym^+_n \, : \, \tr(Y) = 1\}$, where $\tr$ denotes
the trace of a matrix.
The payoff functions of a semidefinite game on $\mathcal{X} \times \mathcal{Y}$
are
\begin{eqnarray*}
  p_A(X, Y) \ = \ \sum_{i,j,k,l} X_{ij} A_{ijkl} Y_{kl}
  \; \text{ and } \;
  p_B(X, Y) \ = \ \sum_{i,j,k,l} X_{ij} B_{ijkl} Y_{kl}
 .
\end{eqnarray*}

If we define the symmetric $m \times m$-matrix $\Phi_A(Y)$ and the symmetric 
$n \times n$-matrix $\Phi'_B(X)$ as
\[
\Phi_A (Y)_{ij} := \sum\limits_{k,l=1}^n A_{ijkl} Y_{kl}
\; \text { and } \;
\Phi'_B (X)_{kl} := \sum\limits_{i,j=1}^m X_{ij} B_{ijkl},
\]
then
the payoffs are
\[
p_A(X,Y) \ = \ \langle X , \Phi_A (Y) \rangle \; \text{and } \;
p_B(X,Y) \ = \ \langle \Phi'_B (X), Y \rangle \, ,
\]
where $\langle \cdot, \cdot \rangle$ denotes the Frobenius dot product.
In \cite{ittv-2023}, the following statement was shown in the more
general context of semidefinite network games.

\begin{thm}
\label{th:nash-sdp1}
The set of Nash equilibria of a semidefinite game are the solutions to
a semidefinite complementarity problem. Namely, a point 
$(X^*,Y^*) \in \mathcal{X} \times \mathcal{Y}$ is a Nash equilibrium
if and only if
\begin{equation}
  \label{eq:cp1}
  \begin{array}{rcl}
  w \cdot I_m - \Phi_A(Y) & \succeq & 0 \, , \\
  v \cdot I_n -  \Phi'_B(X) & \succeq & 0 \, , \\
  \langle X,
  w \cdot I_m - \Phi_A(Y) \rangle & = & 0 \, ,
  \\
  \langle Y, v \cdot I_n -  \Phi'_B(X)\rangle & = & 0 \, ,
  \end{array}
\end{equation}
where $w$ is the payoff of the first player and $v$ is the payoff of the second player.
\end{thm}

For convenience of the reader and to set up the notation,
we recall the proof.

\begin{proof}
A point $(X^*,Y^*)$ is a Nash equilibrium of the game if and only if 
$X^*$ is a best response to $Y^*$ and vice versa, i.e.,
\begin{eqnarray*}
  X^* \in \argmax_{X \in \mathcal{X}} \langle X, \Phi_A(Y^*) \rangle \; \text{ and } \;
  Y^* \in \argmax_{Y \in \mathcal{Y}} \langle \Phi'_B(X^*), Y \rangle \, .
\end{eqnarray*}
The dual of the optimization problem
$
  \max_{X \in \mathcal{X}} \langle X, \Phi_A(Y^*) \rangle
$
is given by
\[
  \min \{ w \, : \, w \cdot I_m - \Phi_A(Y^*)
  \succeq 0, \, w \in \R \}
\]
and the dual of the optimization problem
$\max_{Y \in \mathcal{Y}} \langle \Phi'_B(X^*), Y \rangle$
is given by
\[
  \min \{ v \, : \, -\Phi'_B(X^*)
  + v \cdot I_n \succeq \ 0, \, v \in \R \} \, .
\]
By strong duality, we see that a point 
$(X^*,Y^*) \in \mathcal{X} \times \mathcal{Y}$ is a Nash equilibrium
if and only if the two semidefinite conditions and the two equations
in~\eqref{eq:cp1} are satisfied. The last two equations ensure that the payoffs to the players are 
$w$ and $v$, respectively.
\end{proof}

\subsection{The homotopy view on the Lemke-Howson
  algorithm\label{se:lemke-howson-bimatrix}}

Let $G=(A,B)$ be a bimatrix game with $m \times n$-matrices $A,B$
and denote the mixed strategy spaces by
\[
  \Delta_m \ = \ \{x \in \R^m \, : \, x \ge 0, \; \sum_{i=1}^m x_i = 1\}
  \; \text{ and } \;
  \Delta_n \ = \ \{y \in \R^n \, : \, x \ge 0, \; \sum_{i=1}^n y_i = 1\} \, .
\]

It is useful to mention the following characterization of Nash equilibria
in bimatrix games, which, in turn, we can view as a special case
of Theorem~\ref{th:nash-sdp1}.

\begin{thm}
\label{th:nash2}
The set of Nash equilibria of a bimatrix game are the solutions of
a linear complementarity problem. Namely, a point 
$(x^*,y^*) \in \Delta_m \times \Delta_n$ is a Nash equilibrium
if and only if $ A y \le w \one$ , $B^T x  \le v \one$, and
the following complementarity conditions hold
\begin{equation}
  \label{eq:cp2}
  \begin{array}{rcl}
    \sum\limits_{i=1}^m x_i (w \one - Ay)_i \ = \ 0 \, , \;
    \sum\limits_{j=1}^n y_j (v \one - B^Tx)_j \ = \ 0
  \end{array}  ,
\end{equation}
 where $w$ and $v$ are the payoffs of the first and the second player.
\end{thm}

In the literature, 
the indices in the complementarity conditions
are often viewed as ``labels'' attributed to strategy pairs,
see, e.g.,~\cite[Chapter~45]{von-stengel-handbook}.
Using the notation $M = \{1, \ldots, m\}$ and $N := \{m+1, \ldots, m+n\}$,
the set of labels is $M \cup N$. The indices in $M$ refer to the
$m$ terms in the first sum of~\eqref{eq:cp2} and the indices
in $N$ refer to the $n$ terms in the second sum of~\eqref{eq:cp2}.
We associate a set of labels to every strategy $x$ of the first player, 
which indicates which of the two factors 
of the summands of the complementarity conditions is zero; namely,
\[
  \{ i \in M \, : x_i = 0\} 
  \cup \{m + j \in N \ : \ (v \one - B^T x)_j = 0 \} \, .
\]
Similarly, for a strategy $y$ of the second player, we associate the 
labels
\[
  \{ i \in M \, : \, (w \one - A y)_i = 0\}
  \cup \{ m + j \in N \ : \ y_j = 0 \} \, .
\]
We can also consider the situation
where all the labels in $M$ are associated to the first player.
This corresponds to the artificial strategy $x=0$, which
does not satisfy $\sum_{i=1}^m x_i = 1$. Similarly, we can consider 
the artificial strategy $y=0$ of the second player.
The point $(0,0)$ can be regarded as an artificial equilibrium.
For the artificial strategy 0 of the first
player, always the first factor in each term of the sum is binding,
and similarly for the second player.

We assume non-degeneracy of the bimatrix game and
detail this in the next section.
The Lemke-Howson algorithm starts from the artificial equilibrium
$(0,0)$. 
We fix a strategy for one of the players, say the first, 
by picking $k \in M$.
Then, the $k$-th complementary condition becomes \emph{loose}, 
by omitting the label $k$. We can assume that $k$ refers to
a pure strategy of the first player. Formally, we look for
a solution of~\eqref{eq:cp2} in which the
term with index~$k$ is omitted. In the viewpoint of the labels,
the $k$-th label is dropped. We say that a strategy pair $(x,y)$ is 
\emph{$k$-almost completely labeled} if $(x,y)$ is a solution to the complementarity 
conditions~\eqref{eq:cp2}, where the term $k$ is omitted
from the first sum.

By considering the $k$-almost completely
labeled strategy pairs we define a transition
from the artificial equilibrium to a situation where one of 
the labels of the first player appears twice. Throughout
the algorithm, whenever a label occurs twice, it is dropped
once and this leads to a new pair of strategies where another
(possibly appearing twice) label is taken. As soon as a player 
has moved away from 0, her combinatorial choice corresponds
to a real mixed strategy. Eventually, the algorithm reaches
a pair of strategies that is completely labeled and this pair
is a Nash equilibrium.

The omission of a label can be captured in the homotopy view by
adding a bonus to the $k$-th strategy.
Let $t_0$ be a sufficiently 
large bonus so that if we add $t_0$ to the $k$-th strategy of the first 
player, then this becomes a dominant strategy for the first player.
For $t \in [0,t_0]$, let
$G(t)=(A(t),B(t))$ be the game resulting from $G$ where
we set $a_{kj}(t) = a_{kj} + t$ for $j \in \{1, \ldots, n\}$ and
all other payoff entries in $G(t)$ are taken from $G$.

For the $k$-almost completely labeled points, in the homotopy view
there exists a bonus value $t$ in the homotopy view such that
the $k$-strategy of the first player becomes a best response.
We record this in the following way (see 
\cite[Section 5.4]{govindan-wilson-2003} and \cite[Sections 4.1 and 5]{herings-peeters-2010}).

\begin{prop}
  \label{pr:almost-completely}
	A strategy pair $(x,y)$ is $k$-almost completely labeled if and only if there 
	exists some $t \ge 0$ such that $(x(t),y(t))$ is a Nash equilibrium of 
	the game $G(t)$. 
\end{prop}

The set of solutions of the equations and inequalities 
in Proposition~\ref{pr:almost-completely} defines a set 
$\mathcal{S} \subset (\Delta_m \cup \{0\}) \times 
(\Delta_n \cup \{0\})$ which is a finite union
of polytopes. 

For a non-degenerate game, the polytopes in $\mathcal{S}$ are at
most one-dimensional and they define a graph. Moreover, as a consequence
of the non-degeneracy definition, every vertex of the graph
has at most two adjacent edges \cite{von-stengel-handbook}.
Topologically, $\mathcal{S}$ consists of 
finitely many paths and loops. We call the union of the segments
the \emph{Lemke-Howson paths}.
We call a strategy pair $(x,y)$ an \emph{event point} if
\begin{enumerate}
\item for every $i$ we have $x_i=0$ or $(w \mathbf{1} - Ay)_i=0$ 
  and for every $j$ we have $y_j=0$ or $(v \mathbf{1} - B^T x)_j=0$, and
\item there exists an $i$ with $x_i=(w \mathbf{1} - Ay)_i=0$
  or there exists a $j$ with $y_j = (v \mathbf{1} - B^T x)_j=0$,
\end{enumerate}
where $w$ and $v$ are the payoffs of the first and second player.

In a non-degenerate game, there are only finitely many event points
and they are vertices of the graph. Moreover, the non-degeneracy
definition implies that event points have at most two
incident edges in the graph. Hence, if on the homotopy path
an event point is reached, then there is at most one way to leave it.
The fact that we can leave from a non-Nash equilibrium event point 
can be deduced from combinatorial arguments or from 
general statements on homotopies (see Section~\ref{se:lh-paths}
in a more general context).

\section{An eigenvalue view on semidefinite Nash equilibria\label{se:eigenvalue-view}}

Let $G$ be a semidefinite game and $(X,Y)$ be a Nash equilibrium of $G$.
By Theorem~\ref{th:nash-sdp1}, $(X,Y)$ satisfies the complementarity conditions 
\begin{eqnarray}
  \langle X, w \cdot I_m - \Phi_A(Y) 
    \rangle & = & 0, \label{eq:scalar-prod-cond1}  \\
  \text{ and } \ \langle Y, v \cdot I_n - \Phi'_B(X)
    \rangle & = & 0 \, .
   \label{eq:scalar-prod-cond2}
\end{eqnarray}
The inner product conditions~\eqref{eq:scalar-prod-cond1} 
and~\eqref{eq:scalar-prod-cond2} are equivalent to the matrix equations 
\begin{eqnarray*}
  X ( w \cdot I_m - \Phi_A(Y)) 
    \ = \ 0 \, \text{ and } 
  Y  (v \cdot I_n - \Phi'_B(X)) \ = \ 0 \, .
\end{eqnarray*}
Namely, it is well known (see, for example, \cite{aho-1997}) that if
two matrices $S, T \in \sym_n^+$ satisfy
$\tr(ST) = 0$, then $S^{1/2} T S^{1/2} \in \sym_n^+$ and
$\tr(S^{1/2} T S^{1/2}) = \tr(ST) = 0$, which implies
$S^{1/2} T S^{1/2} = 0$ and thus $ST = 0$.
Note that the product of two positive semidefinite matrices has 
nonnegative eigenvalues, but in general the product is not a symmetric 
matrix.

Using the abbreviations
$W=W(Y) := w \cdot I_m - \Phi_A(Y) $
and $V = V(X) :=v \cdot I_n - \Phi'_B(X)$,
we can write the system of matrix equations as
\begin{eqnarray*}
  X \cdot W(Y) \ = \ 0 \; \text{ and } \;
  Y \cdot V(X) \ = \ 0 \, .
\end{eqnarray*}

If two real $n \times n$-matrices $U,V$ 
satisfy $U V = 0$, then the sum of the multiplicities of the eigenvalue 0 
in $U$ and in $V$ is at least $n$.
This follows from Sylvester's rank inequality 
$\rank (UV) \ge \rank(U) + \rank(V) - n$, e.g., \cite{marcus-mink-1992},
hence in our case $\rank(U) + \rank(V) \le n$.

We transfer the underlying condition "$x$ has certain labels"
from the bimatrix games to an eigenvalue condition.
Since $X \cdot W = 0$, we also have $W \cdot X = (X^T \cdot W^T)^T
= (X \cdot W)^T = 0$, that is, the matrices $X$ and $W$ commute. Hence,
there exists a common system of eigenvectors (see \cite{horn-johnson-2012} or, 
in our context, \cite[Lemma~3]{aho-1997}). Thus, there exists an orthogonal matrix
$Q \in \R^{n \times n}$, with $Q^T Q = I$, such that
\begin{eqnarray*}
  X & = & Q^T \diag(\lambda_1(X), \ldots, \lambda_m(X)) Q \, , \\
  W & = & Q^T \diag(\lambda_1(W), \ldots, \lambda_m(W)) Q \, ,
\end{eqnarray*}
where $\lambda_1(X), \ldots, \lambda_m(X)$ and 
$\lambda_1(W), \ldots, \lambda_m(W)$ are the eigenvalues
of $X$ and of $W$.
In other words, the matrices are simultaneously diagonalizable.
Notice that we cannot assume that both
sequences of eigenvalues are decreasing. 
We have
\begin{equation}
  \label{eq:xw1}
  X W \ = \ Q^T \diag(\lambda_1(X) \lambda_1(W), \ldots,
    \lambda_m(X) \lambda_m(W)) Q \, .
\end{equation}

Hence, at a Nash equilibrium, $(X,Y)$, we find that
for each common eigenvector $v_i$ of $X$ and $W$, at least one of
the eigenvalues $\lambda_i(X)$ or $\lambda_i(W)$ is zero. 
A similar statement is true for $Y$ and $V$. 
These statements, combined with the degeneracy condition discussed in the sequel, lead to the conclusion that the sum of the multiplicities of the eigenvalue 0 in $X$ 
and $W$ is $m$.

\subsection*{Non-degeneracy of semidefinite games}

In the \emph{generic situation}, if $X,Y$ is a Nash equilibrium and $W$, 
$V$ are defined as above, then 
the sum of the multiplicities of 0 as an eigenvalue of $X$ and $W$
is $m$ and 
and for $Y$ and $V$
is $n$.
However, from Pataki's rank results for semidefinite programming \cite{pataki-1998}, the
situation is more complicated than in the case of linear programming in the following sense.
If a linear program is primally and dually non-degenerate,
then strict complementarity holds. For semidefinite programs
this is not true and to capture this, the theory of strict complementarity
for semidefinite programs was developed (see \cite{aho-1997}). 

\begin{definition}
Let $U, V \in \sym_n$ with $UV = 0$. We say that 
\emph{strict complementarity}
holds if
\[
  \rank U + \rank V = n \, ,
\]
i.e., if and only if for every $i \in \{1, \ldots, n\}$, exactly one of the
conditions $\lambda_i(U) = 0$ or 
$\lambda_i(V) = 0$ is satisfied. Here, we used the earlier
notation $\lambda_i(\cdot)$ from the simultaneous diagonalization.
\end{definition}

\begin{definition}
Let $(X,Y)$ be a Nash equilibrium of the semidefinite game $G$ on $\mathcal{S}_m \times \mathcal{S}_n$
and let $W=W(X)$ and $V=V(Y)$ be as defined previously. We say that the Nash equilibrium
$(X,Y)$ satisfies the \emph{strict complementarity} condition if
\[
  \rank X + \rank W = m \: \text{ and } \: \rank Y + \rank V = n \, .
\]
\end{definition}

For a bimatrix game $(A,B)$ on $\Delta_m\times \Delta_n$ we have the following 
equivalent conditions to define a non-degenerate game (see, e.g.,
\cite{von-stengel-handbook}). We use the labels $M=\{1, \ldots, m\}$
and $N=\{m+1, \ldots, n\}$ from Section~\ref{se:prelim}
and denote by
$\supp x := \{i \, : \, x_i \neq 0\}$ the support of a mixed strategy $x$. 
Every point in $\Delta_m \cup \{0\}$ and every point in 
$\Delta_n \cup \{0\}$ carries a set of
labels in $M \cup N$. For $i \in M$ and $j \in N$, set
\[
  \begin{array}{rcl}
  X(i) & = & \{x \in \Delta_m \, : \ x_i = 0 \} \, , \\ [0.4ex]
  X(m+j) & = & \{ x \in \Delta_m \, : \, j = \argmax_ {j'\in [n]}(Bx)_{j'} \} \, , \\ [0.4ex]
  Y(i) & = & \{ y \in \Delta_n \; \, : \, i = \argmax_ {i'\in [m]}(Ay)_{i'} \} \, , \\ [0.4ex]
  Y(m+j) & = & \{y \in \Delta_n \; \, : \ y_j = 0 \} \, .\\      
  \end{array}
\]
Note that for $j \in [n]$, the set 
$X(m+j)$ is the set of strategies of the first player for which the $j$-th pure strategy of the 
second player is a best response. Similarly, for $i \in [m]$, the set
$Y(i)$ is the set of strategies of the second player for which the $i$-th pure strategy of the
first player is a best response. A bimatrix game is called
\emph{non-degenerate} if one of the following equivalent conditions
is satisfied.

i) For any mixed strategy $x$ of the first player, the second player has at most
  $|\supp(x)|$ best pure responses and vice versa.
  
ii) For any $x \in \Delta_m$ with a set of labels $K \subset [m + n]$ and 
$y \in \Delta_n$ with a set of labels $L \subset [m + n]$,
  the set 
  $\bigcap_{k \in K} X(k)$ has dimension $m-|K|$ and the set
  $\bigcap_{l \in L} Y(l)$ has dimension $n-|L|$.

\medskip

The non-degeneracy notion can be seen as formalizing that certain 
intersections have the expected
dimension. Note that even for bimatrix games, non-degeneracy questions
often involve nontrivial aspects. Proving that some of the various 
non-degeneracy definitions in the literature are equivalent requires
nontrivial proofs (see \cite{von-stengel-handbook}).
From an algorithmic point of view,
even deciding whether i), or equivalently ii), is satisfied in a sparse game 
is an NP-hard problem \cite{du-2013}.
For our purposes, the following non-degeneracy condition for semidefinite games
is useful:

\begin{definition}
\label{de:nondegen}
A semidefinite game is \emph{non-degenerate} if the following two conditions are satisfied:

I) For any pair of mixed strategies $X,Y$ of rank $k_1$ and $k_2$ respectively, 
  we have $\rank V(X) \ge n-k_1$ and $\rank W(Y) \ge m-k_2$.
  
II) The intersection of the two varieties defined by $X \cdot W(Y) = 0$ and by $Y \cdot V(X) = 0$ is zero-dimensional.
\end{definition}

Here, the varieties in II) are considered as varieties over the complex
numbers. 
In a Nash equilibrium, this implies that $\rank V(X) = n-k_1$ and $\rank W(Y) = m-k_2$ and furthermore that the strategies $X$ and $Y$ must be of equal rank $k=k_1=k_2$.

\medskip

For the special case of diagonal games, which correspond to bimatrix games,
the conditions i) and I) are equivalent and
condition ii) implies II). Namely, for II) we have to consider
only $m$ equations, where the left-hand side of each equation is a product
of a variable and an affine-linear form. The complex algebraic variety is the 
union of linear varieties. Using the projective view (see, e.g.,
\cite{von-stengel-handbook}), the additional
variable can be normalized to 1, where we have to assume without loss of generality
that all entries in the payoff matrices are positive. 
Hence, in the special case of diagonal games,
our definition is equivalent to the non-degeneracy definition for bimatrix games.
Note that while condition i) is equivalent to ii) for each bimatrix game, 
condition I) does not imply II) for each semidefinite game, see 
Example \ref{ex:non-deg}. 

The two conditions I) and II) are needed for our treatment of the Lemke-Howson
paths. As a consequence of statement II), the number of Nash equilibria is finite.
Namely, the zero-dimensionality means that the variety consists of a finite number of complex points, and hence the number of (real) density matrices satisfying the 
equations is finite. 

We remark that our non-degeneracy assumptions are stronger than just requiring
the game to have a finite number of Nash equilibria. This is analogous to the common non-degeneracy notions for bimatrix games (see \cite{von-stengel-handbook}). 
Further note that earlier results of Bich and 
Fixary \cite{bich-fixary-2024} (in a more general semialgebraic setting)
imply that in the generic case, the number of Nash equilibria in a semidefinite
game is finite and odd. From the viewpoint of complexity, deciding whether
a semidefinite game in sparse encoding is non-degenerate is at least NP-hard,
since already the decision problem for the special case of diagonal games is 
an NP-hard problem.

\begin{lemma}
In a non-degenerate semidefinite game, any Nash equilibrium satisfies strict complementarity.
\end{lemma}
\begin{proof}
Consider a Nash equilibrium $(X,Y)$ in a non-degenerate semidefinite game. Let $W=W(Y)$ and $V=V(X)$ be defined as above and let $\rank X = \rank Y = k$. Since $W\cdot X = 0$, we know that $\rank W\le m-k$. Furthermore, any matrix $X'$ that satisfies $W \cdot X'=0$ is a best response to $Y$ by construction. Since the game is non-degenerate any such matrix must be of rank at most $k$ and therefore $\rank W\ge m-k$. Therefore, $\rank W = m-k$ and likewise $\rank V = n-k$.
\end{proof}

\begin{example}\label{ex:five_NE}
a) For given $i,j \in \{1, \ldots, m\}$, denote by $A_{ij..}$ the 
slice $(A_{ijrs})_{1 \le r,s \le n}$ of the payoff tensor $A$.
Consider the semidefinite game given by
\[
  A_{11..} = \left( \begin{array}{cc}
    1 & 0 \\ 0 & 0
  \end{array} \right), \;
    A_{12..} = \left( \begin{array}{cc}
    0 & c \\ c & 0
  \end{array} \right), \;
    A_{21..} = \left( \begin{array}{cc}
    0 & c \\ c & 0
  \end{array} \right), \;
    A_{22..} = \left( \begin{array}{cc}
    0 & 0 \\ 0 & 1
  \end{array} \right)
\]
with some constant $c > \frac{1}{2}$, and $B_{ij..} = A_{ij..}$ for $1 \le i, j \le 2$. There are five
Nash equilibria (see \cite{itt-2024}):
\begin{eqnarray*}
  && X = Y = \begin{pmatrix} 1 & 0 \\ 0 & 0 \end{pmatrix}, \;
  X = Y = \begin{pmatrix} 0 & 0 \\ 0 & 1 \end{pmatrix}, \;
  X = Y =  \begin{pmatrix} 1/2 & 0 \\ 0 & 1/2 \end{pmatrix}, \\
  & & X=Y= \begin{pmatrix} 1/2 & 1/2 \\ 1/2 & 1/2 \end{pmatrix}
  \; \text{ and } \;
  X=Y = \begin{pmatrix} 1/2 & -1/2 \\ -1/2 & 1/2 \end{pmatrix} \, .
\end{eqnarray*}
All of these five Nash equilibria are strict complementary.

\smallskip

b) Now consider the variation where $a_{2222}$ and $b_{2222}$ are replaced by
zero, i.e.,
\[
  A_{11..} = \left( \begin{array}{cc}
    1 & 0 \\ 0 & 0
  \end{array} \right), \;
    A_{12..} = \left( \begin{array}{cc}
    0 & c \\ c & 0
  \end{array} \right), \;
    A_{21..} = \left( \begin{array}{cc}
    0 & c \\ c & 0
  \end{array} \right), \;
    A_{22..} = \left( \begin{array}{cc}
    0 & 0 \\ 0 & 0
  \end{array} \right)
\]
with $c > 1/2$, and $B_{ij..} = A_{ij..}$ for $1 \le i, j \le 2$. Then the payoff of each player is
$
  p(X,Y) \ = \ x_{11} y_{11} + 4 c x_{12} y_{12} \, .
$
Since $c > 1/2$, it is possible that the payoffs of each player 
become larger than 1.
For the Nash equilibria $(X,Y)$ with $x_{12}=0$, we obtain
\begin{eqnarray*}
  && X=Y=\left( \begin{array}{cc}
    1 & 0 \\
    0 & 0
  \end{array} \right), \;
  \text{ where }
  W=V=\left( \begin{array}{cc}
  0 & 0 \\
  0 & 0
  \end{array} \right) , \\
  & \text{as well as} & X =Y=\left( \begin{array}{cc}
  0 & 0 \\
  0 & 1
  \end{array} \right), \;
  \text{ where }
  W=V=\left( \begin{array}{cc}
  0 & 0 \\
  0 & 0
  \end{array} \right),
 \end{eqnarray*}
and thus these two Nash equilibria are not strictly complementary.

In the case $x_{12} \neq 0$, we can assume positive signs for the non-diagonal elements of
both players as well as $x_{12} = \sqrt{x_{11} x_{22}}$ and $y_{12} = \sqrt{y_{11} y_{22}}$.
Hence, the payoffs are
$
  p(X,Y) \ = \ x_{11} y_{11} + 4 c \sqrt{x_{11}(1-x_{11})} \sqrt{y_{11} (1-y_{11})}.
$
In a Nash equilibrium, the partial derivatives
\begin{align*}
  p_{x_{11}} = & \ y_{11} +
  \frac{2 c \sqrt{y_{11} (1-y_{11})} (1-2 x_{11})}{\sqrt{x_{11} (1-x_{11})}}, \quad
  p_{y_{11}} = \ & x_{11} +
  \frac{2 c \sqrt{x_{11} (1-x_{11})} (1-2 y_{11})}{\sqrt{y_{11} (1-y_{11})}}
\end{align*}
of $p(X,Y)$ vanish. This system gives exactly
one solution, $x_{11} = y_{11} = \frac{2c}{4c-1}$.
Altogether, in the case $x_{12} \neq 0$,
we obtain the two Nash equilibria
\[
  X = Y = \begin{pmatrix} 
    \frac{2c}{4c-1} & \frac{\sqrt{2c(2c-1)}}{4c-1} \\
    \frac{\sqrt{2c(2c-1)}}{4c-1} & \frac{2c-1}{4c-1}
  \end{pmatrix}
  \; \text{ and } \;
   X = Y = \begin{pmatrix} 
    \frac{2c}{4c-1} & -\frac{\sqrt{2c(2c-1)}}{4c-1} \\
    -\frac{\sqrt{2c(2c-1)}}{4c-1} & \frac{2c-1}{4c-1}
  \end{pmatrix}
\]
which are both strict complementary.

Note that the total number of Nash equilibria of the game is not odd, but it is even. To illuminate this degeneracy, we note that the related bimatrix game with payoff matrices $\begin{pmatrix} 1 & 0 \\ 0 & 0 \end{pmatrix}$,
$\begin{pmatrix} 1 & 0 \\ 0 & 0 \end{pmatrix}$ has an even number of Nash equilibria as well, namely two.
\end{example} 

\begin{example}\label{ex:non-deg}
Consider the semidefinite game given by $\Phi_A(Y)=Y$ and $\Phi_B'(X)=X$. This is exactly the previous Example \ref{ex:five_NE} a) with $c=\frac{1}{2}$.
This game satisfies the non-degeneracy condition I) but not II). Indeed, for any $t\in [0,1]$ the pair $$X=Y=\begin{pmatrix}
t & \sqrt{t(1-t)}\\
\sqrt{t(1-t)} & 1-t
\end{pmatrix}$$ forms a Nash equilibrium.
The corresponding matrices $$W(Y)=V(X)=\begin{pmatrix}
1-t & -\sqrt{t(1-t)}\\
-\sqrt{t(1-t)} & t
\end{pmatrix}  $$ are of rank one.
\end{example}

Nash equilibria satisfying strict complementarity provide a good situation.
Indeed, we assume
that for all but finitely many $t$ (were $t \ge 0$), the Nash equilibria in the game 
$G(t)$ satisfy strict complementarity. The exceptional values for $t$ induce
the event points.
For a matrix $M$, denote by $\det_{k} M$ the vector of all $k \times k$-minors of $M$.

\begin{lemma}
\label{le:equations1}
Let $(X(t),Y(t))$ be an event point which is not a Nash equilibrium. Then
there exists some $t > 0$ such that
\begin{enumerate}
\item $(X(t),Y(t))$ is satisfies the equations.
$X(t) \cdot W(Y(t))  =  0$ \text{ and } $Y(t) \cdot V(X(t)) = 0$.
\item There exists some $k \in \{1, \ldots, m\}$ with
\[
  {\det}_k(X(t)) = 0, \, {\det}_{m-k+1}(W(Y(t))) = 0
\]
or there exists some $l \in \{1, \ldots, n\}$ with
\[
  {\det}_{n-l+1}(Y(t)) = 0, \, {\det}_{l}(V(X(t))) = 0 \, .
\]
\end{enumerate}
\end{lemma}

\begin{proof}
Assume that the event point is not a Nash equilibrium. Then
in every product of two corresponding eigenvalues, at least
one of the is zero, and there exists a product where both of
the eigenvalues are zero. Hence, 
$\rank X(t) + \rank W(Y(t)) < m -1$ or 
$\rank Y(t) + \rank V((X)) < n - 1$. This translates into
the determinantal conditions of the statement of the lemma.
\end{proof}

\section{Lemke-Howson paths\label{se:lh-paths}}

The framework of Lemke-Howson paths for bimatrix games presented in Section~\ref{se:prelim}
generalizes to semidefinite games. We begin from the homotopy view described
in Section~\ref{se:prelim}. Choosing the $k$-th strategy in the bimatrix setting
carries over to choosing the $k$-th \emph{diagonal strategy} for some $k$ and 
adding the bonus means to increase the payoff entries
$a_{kkrr}$ by some bonus value $t$ for every $r$.
That is, every entry of the slice $(A_{kkrs})_{1 \le r,s \le n}$ of the payoff tensor $A$ 
is increased by the bonus value $t$.
The homotopy is based on the following general result from
Mas-Colell \cite{mascolell-1974}, see also Herings and Peeters
\cite{herings-peeters-2010}. 

\begin{thm}
\label{th:general-homotopy}
Let $C \neq \emptyset$ be a compact, convex subset of $\R^d$
and let $H:[0,1] \times C \to C$ be an upper hemicontinuous correspondence
which is nonempty and convex-valued. Further, let
\[
  F_H \ = \ \{(\lambda,x) \in [0,1] \times C \, : \ x \in H(\lambda,x)\}
\]
be the set of fixed points of $H$. Then, $F_H$ contains a connected
set $G_H$ such that $(\{0\} \times C) \cap G_H \not= \emptyset$
and $(\{1\} \times C) \cap G_H \neq \emptyset$.
\end{thm}

For a degeneracy discussion, see the consideration 
after \cite[Theorem 2]{herings-peeters-2010}.
Further, 
as described in \cite{herings-peeters-2010} in the generic case, we
obtain a finite collection (in the topological sense) of arcs and loops.

For semidefinite games, we consider the parameter values 
$[t_0, 0]$ for some sufficiently large $t_0$, and we can map that
interval to $[0,1]$. The correspondence 
$H : [t_0,t] \times (\mathcal{X} \times \mathcal{Y}) 
\to (\mathcal{X} \times \mathcal{Y})$ is given by
\[
  H(t,(X,Y)) \ = \ \beta_1(t,(X,Y)) \times \beta_2(t,(X,Y))
\]
where 
\begin{eqnarray*}
\beta_1(t,(X,Y)) 
& = & \argmax_{\bar{X} \in \mathcal{X}} \langle \bar{X}, \Phi_A(Y)\rangle \, , \\
\beta_2(t,(X,Y))
& = & \argmax_{\bar{Y} \in \mathcal{Y}} \langle \Phi_B'(X), 
\bar{Y} \rangle
\end{eqnarray*}
are the best response correspondences of the two players.

\begin{lemma}
$H$ is upper hemicontinuous.
\end{lemma}

\begin{proof}
Recall that a correspondence 
$F$ between two topological spaces is called \emph{com\-pact-valued}
if for every $x \in A$ the set $F(x)$ is compact. 
A compact-valued correspondence is upper-hemicontinuous at
a point $a \in A$ if and only if for every sequence
$(a_n) \to a$ and every sequence $(b_n) \in B$ 
with $b_n \in F(a_n)$ for all $n$, there exists a
converging subsequence of $(b_n)$ whose limit point $b$
is in $F(a)$ (see, e.g., \cite[Corollary 17.17]{aliprantis-border-2006}). 
Since the best response correspondence is compact-valued, 
the characterization implies that $H$ is upper hemicontinuous.
\end{proof}

Let $t_0$ be the sufficiently large bonus added to the $k$-th diagonal
strategy of player~1. That is, we set $a_{kkrr}(t) = a_{kkrr} + t$ for
$r \in \{1, \ldots, n\}$.
For $t \in [0,t_0]$, we consider the set of the game $G(t)$, where $G(t)$ is
defined by adding the bonus $t$ to the $k$-th strategy of player~1.

As a generalization of the Lemke-Howson algorithm, we describe a
sequence of (nonlinear) curve segments and explain how the combinatorics
changes in the endpoints of the curve segments.
For each $t$, we consider the complementarity conditions
$X(t) W(t) = 0$ and $Y(t) V(t) = 0$. Considering the solution 
from Theorem~\ref{th:general-homotopy}, as $t$ varies, the
eigenvectors continuously vary with $t$. In generalization of
the Lemke-Howson algorithm for bimatrix games, an event point 
occurs if one of the eigenvalues additionally reaches 0 or if $t$ reaches zero.
In the case of reaching $t=0$, we have found a Nash equilibrium for our
original problem.
Note that, when interpreting the parameter $t$ as time, tracing
the homotopy might require us to consider larger values of $t$
(see \cite{herings-peeters-2010}).

At the parameter value $t_0$, player 1 plays the $k$-th 
diagonal strategy. Player 2 plays the best response $s$ to the 
$k$-th diagonal strategy of the first player. Now, let $t$ decrease from
$t_0$. As long as $t$ is still sufficiently large, player~1 still plays the
$k$-th diagonal strategy and thus player~2 still plays $s$. Hence,
at the beginning of the homotopy, as long as $t$ is sufficiently large,
the strategies of both players stay constant.

If an additional eigenvalue reaches zero, say, for example, in the local
situation $\lambda_i(X(t)) = 0$ on a branch, an eigenvalue $\lambda_i(W(\bar{t}))$ 
becomes additionally zero for some value $\bar{t}$, 
then an event point occurs. In this point, 
strict complementarity is not satisfied. In order to continue the path,
non-negativity of the eigenvalues must be preserved, so that for $t$ infinitesimally
larger than $\bar{t}$, the eigenvalue $\lambda_i(W(t))$ will stay zero and
the eigenvalue $\lambda_i(X(t))$ will become positive. Note that the joint
eigenvectors will continuously change with $t$.

To concentrate on the main ideas, we consider an appropriate
non-degeneracy. We assume that at every event point,
at most two curve branches leave from that event point.
A finiteness result on the number of event points and thus on the number
of curve branches can be obtained as a consequence of Lemma~\ref{le:equations1}.

\begin{cor}
\label{cor:finiteness}
Assume that for all $k \in \{1, \ldots, m\}$ and $l \in \{1, \ldots, n\}$
the systems
\[
 {\det}_k(X(t)) = 0, \, {\det}_{m-k+1}(W(Y(t))) = 0
\]
and
\[
  {\det}_{n-l+1}(Y(t)) = 0, \, {\det}_{l}(V(X(t))) = 0 \,
\]
have a finite number of real solutions 
$(X,Y,w,v,t) \in \mathcal{X} \times \mathcal{Y} \times \R \times \R \times \R_{>0}$. 
Then, the number of event points is finite.
\end{cor}

\begin{proof}
By Lemma~\ref{le:equations1}, for every event point which is not a Nash equilibrium
there exists some $t > 0$ such that one of the systems has a real solution.
\end{proof}

In generalization of the bimatrix situation, the event points together
with the nonlinear curve segments connecting the event points define
a combinatorial graph. Since already in the case of $n \times n$ bimatrix games,
Lemke-Howson paths can be exponentially long (see \cite{savani-von-stengel-2006}),
there cannot be a polynomial bound on the number of event points for semidefinite games.

\section{Tracing the Lemke-Howson paths\label{se:tracing}}

To trace the curve segments in the regions of $t$ free of event points,
we exploit the fact that these curve segments are smooth
and we can employ a predictor-corrector method, based on the implicit function theorem and Newton's operator
to trace them.
There is a lot of related work on these subjects
as this problem appears in different contexts.
For example, in the numerical algebraic geometry community,
curve tracing is the main operation of homotopy continuation
algorithms, e.g., \cite{WamSom-numalg-bk-11,bshw-bertini-13}.
Also, they are commonly treated in the community of geometric modeling, e.g., 
\cite{FauMic-trace-07,Georg-trace-81,Taubin-curves-93}.
Starting from a regular point on the curve, say $p$, we compute a vector tangent, say $t$, and we compute the point $p + \eta \tfrac{t}{\|t\|}$, 
where $\eta$ is the step size; this is the prediction phase. 
If $\eta$ is small, then the new point is close to the curve
and we correct it using a variant of Newton's method.
We should also note that other methods also exist, especially for real curves, e.g.~\cite{Arnon-curve-83,Taubin-curves-93}.

The event points correspond to singularities of the curve.
Many branches can go through an event point.
To understand the local geometry around the event points
and to select the correct branch for the Lemke-Howson path
we employ the fact that locally, that is, around the singularity, 
we can parametrize the curve branches using Puiseux series, e.g., \cite{bpr-2006,Walker-curves-78}. 
This computation raises many numerical challenges, e.g., 
\cite{Duval-Q-Puiseux-89,PotRab-Qp-Puiseux-11,verschelde-2008},
and the same is true for the numerical tracing of the Lemke-Howson paths
close and so not close to the event points.
We leave the algorithmic details of these computations for a future communication.
In the sequel, we present a short introduction to Puiseux series and expansions
and explicitly work out
symbolically an example of tracing the Lemke-Howson path,
to give an overview of the various computations.

\subsection{Parametrization of the branches of a curve and Puiseux series}
\label{sec:puiseux}

Consider a plane algebraic curve $\cC$, defined
implicitly as the zero locus 
of the equation $f(x, y) = 0$, where $f \in \C[x, y]$ 
is a polynomial in two variables with complex coefficients
of degree $n$ w.r.t.\ $y$.
Further assume a point on $\cC$ which, without loss of generality, we can assume (after translation) to be the origin. 
If $\cC$ is smooth at 0, then we can exploit the implicit function theorem and we can locally parametrize the unique branch going through 0.
However, if 0 is not a smooth point, that is, if it is a singular point, then several branches might go through it. Moreover, 
the implicit function theorem is not applicable anymore. 

Under some mild assumptions, we can compute the number of branches going through the origin
and also find a local parametrization for them, 
by computing solutions $y(x)$ for $f$; that is, we interpret $f$ as a 
polynomial in $y$ with coefficients polynomials in $x$. 
The solutions are of the form $y = \sum_{i=1}^{\infty} a_i x^{i/n}$, 
and this representation leads to a parametrization 
of the branches of the form
\[
	x = t^n, \quad y = \sum\nolimits_{i=1}^{\infty} a_i t^i,
\]
where $t$ is a new variable; the parametrization 
is a formal power series.

We compute a parametrization $(t^n, \phi(t))$
using Newton's algorithm, where $n\ \in \N$, $\phi\in\C[[t]]$
and $\C[[t]]$ denotes the ring of power series in $t$ with complex coefficients.
One can prove that they are locally convergent
and that 
\[
    f(t^n, \phi(t))\,=\,0 \quad \text{ in } \C[t].
\]
We can rewrite the previous equation as 
$f(x, \phi(y^{1/n})) = 0$,
which allows to deduce that if we consider $f$ as a polynomial in $y$ with coefficients in $\C[x]$, then 
$\phi(x^{1/n})$ is a \emph{root} of $f$ in the larger ring $\C[[x^{1/n}]]\supset\C[x]$.

The value of $n$ can vary for different polynomials.
Even more, the various branches going through a singular point have different parametrizations, therefore, they are associated with different values of $n$; we call this the ramification index.
This leads to the introduction of the more general ring of Puiseux series
\[
    \C[[x^{\star}]]\,:=\,\displaystyle\bigcup_{n\in\mathbb{N}}\C[[x^{1/n}]].
\]
The quotient field of this ring, $\C\langle\langle x\rangle\rangle$,  
the \emph{field of Puiseux series}, consists of Puiseux series having, eventually, finitely many negative exponents. 
The field of Puiseux series is algebraically closed. 
In particular, $f$ splits over $\C\langle\langle x\rangle\rangle$ in linear factors and thus has $\deg_y(f)$ formal parametrizations, counted with multiplicities.
We refer the interested reader to the classical literature 
on algebraic curves for further details
\cite{casas-sing-curves-00,casas-bn-19,Walker-curves-78}.

In the case where the polynomial that defines the curve $\cC$ has integer (or rational) coefficients, that is $f \in \Z[x,y]$, 
there is a variant of Newton's algorithm that allows us to compute rational Puiseux series \cite{Duval-Q-Puiseux-89};
there are even effective bit complexity bounds \cite{PotRab-Qp-Puiseux-11,Walsh-Q-Puiseux-99}
for the corresponding algorithms. 
Hence, all the computation remain in the rationals.
When we are interested only in the real branches, then it is possible to exploit more geometric methods to determine their number,
e.g., \cite{ddrrs-dcg-22}.

If the curve lives in the $\C^d$ (or $\R^d$), where $d \geq 3$,
then it is still possible to compute the Puiseux expansions of 
its branches through a singular point by first projecting the curve to the plane and then lift the parametrization
\cite{Teissier-bias-intro-07}.

\subsection{An example on tracing}

\begin{ex}
To keep notation simple,
we consider a hybrid game, where the strategy space of the first
player is a simplex with two pure strategies and the strategy space
of the second player is the set of $2 \times 2$ real density matrices. To keep it in the framework of semidefinite games,
we can also write this, as in the following, as a $2\times 2\times 2\times 2$
semidefinite game, but we enforce $x_{12}=x_{21}=0$. For given $i,j$, we write $A_{ij..}$ for the slice 
$(A_{ijrs})_{1 \le r,s \le n})$ of the tensor $A$. Let
\[
  A_{11..} = \left( \begin{array}{cc}
    1 & 0 \\ 0 & 1
  \end{array} \right), \;
    A_{12..} = \left( \begin{array}{cc}
    0 & 0 \\ 0 & 0
  \end{array} \right), \;
    A_{21..} = \left( \begin{array}{cc}
    0 & 0 \\ 0 & 0
  \end{array} \right), \;
    A_{22..} = \left( \begin{array}{cc}
    2 & 2c \\ 2c & 2
  \end{array} \right)
\]
and
\[
    B_{11..} = \left( \begin{array}{cc}
    2 & 0 \\ 0 & 1
  \end{array} \right), \;
    B_{12..} = \left( \begin{array}{cc}
    0 & 0 \\ 0 & 0
  \end{array} \right), \;
    B_{21..} = \left( \begin{array}{cc}
    0 & 0 \\ 0 & 0
  \end{array} \right), \;
    B_{22..} = \left( \begin{array}{cc}
    2 & c \\ c & 1
  \end{array} \right).
\]
Specifically, we consider the choice $c=1/10$.
In the homotopy setting, we add $t$ to each entry of 
$A_{11..}$ so that $A_{11..}$ then becomes
\[
 A_{11..} = \left( \begin{array}{cc}
    t+1 & t \\ t & t+1
  \end{array} \right) .
\]
For large $t$, in particular $t > 1$ (see below), the unique Nash equilibrium is
\begin{equation}
  \label{eq:xyidentity}
    X = \left( \begin{array}{cc}
    1 & 0 \\ 0 & 0
  \end{array} \right), \;
    Y = \left( \begin{array}{cc}
    1 & 0 \\ 0 & 0
  \end{array} \right) \, .
\end{equation}
For decreasing values of $t$ with $t > 1$, this Nash equilibrium remains
the unique Nash equilibrium for the game $G(t)$. We have 
\[
    V(t) = \left( \begin{array}{cc}
    0 & 0 \\ 0 & 1
  \end{array} \right), \;
    W(t) = \left( \begin{array}{cc}
    0 & 0 \\ 0 & t-1
  \end{array} \right) \, .
\]

For $t=1$, an event point occurs and the combinatorics changes. 
Namely, the matrix $W$ becomes the zero matrix and for 
$t$-values slightly smaller than 1, the matrix $W$ stays the zero matrix.
For $t$-values slightly smaller than 1, the matrix $X$ is a rank 2 matrix.
Indeed, in the course of decreasing $t$ from 1,
the left upper entry of $X$ becomes smaller and the right lower
entry becomes larger.

Locally at $t=1$, besides the solution~\eqref{eq:xyidentity},
the following branches exist locally in the event point. 
We substitute $t = s+1$. Note that for $t < 1$ we have $s < 0$.
\begin{eqnarray*}
  X^{(1)} & = & \left( \begin{array}{cc}
    \frac{20s+8+25 s \sqrt{10s+4}}{4 (5s+2)}
     & 0 \\ [0.7ex]
    0 & - \frac{25s}{2 \sqrt{10s+4}}
  \end{array}
  \right), \\
  Y^{(1)} & = & \left( \begin{array}{cc}
    \frac{5(\frac{1}{2}s+\frac{2}{5}
    +\frac{1}{5}\sqrt{10s+4})}{5s+4} & 
    -\frac{5s}{2(5s+4)} \\ [0.7ex]
    -\frac{5s}{2(5s+4)} &  
    \frac{5(\frac{1}{2}s 
    +\frac{2}{5}
    -\frac{1}{5} \sqrt{10s+4})}{5s+4}
  \end{array}
  \right), \, \\
  v & = & \frac{60s+24+\sqrt{250s^3+500s^2+320s+64}}{8(5s+2)} \, , \quad
  w \ = \ \frac{9s+8}{5s+4} \, .
\end{eqnarray*}

The solutions are the solutions of the parametric systems of equations, 
obtained symbolically. The other real branches leave the strategy
spaces, as in the case of bimatrix games.

We give the corresponding Puiseux series.
This has to be seen as a vector of Puiseux series, in the sense that the components
have to fit together. For example, as it is the case here, there might be several Puiseux series 
for $w$ which are candidates by just looking at the symbolic expression for $w$;
in the example, $w$ can be computed from $Y$.
\begin{small}
\begin{eqnarray*}
  X^{(1)} & = & \left( \begin{array}{cc}
    1
    {+}\frac{25}{4}s 
    {-}\frac{125}{16} s^2
    {+}\frac{1875}{128} s^3        
    {-}\frac{15625}{512}s^4
    {+} \cdots
     & 0 \\ [0.7ex]
    0 & 
    {-}\frac{25}{4}s 
    {+}\frac{125}{16}s^2
    {-}\frac{1875}{128}s^3
    {+}\frac{15625}{512}s^4
    {+} \cdots
  \end{array}
  \right), \, \\
  Y^{(1)} & = & \left( \begin{array}{cc}
    1
   {-}\frac{25}{64}s^2 
    {+}\frac{125}{128}s^3
    {-}\frac{8125}{4096}s^4
    {+} \cdots & 
    {-}\frac{5}{8}s 
    {+}\frac{25}{32}s^2
    {-}\frac{125}{128}s^3
    {+}\frac{625}{512}s^4
    {+} \cdots \\
    [0.7ex]
    {-}\frac{5}{8}s 
    {+}\frac{25}{32}s^2
    {-}\frac{125}{128}s^3    
    {+}\frac{625}{512}s^4
    & 
     \frac{25}{64}s^2 
    {-}\frac{125}{128}s^3
    {+}\frac{8125}{4096}s^4
    {+} \cdots
  \end{array}
  \right)
\end{eqnarray*}
\end{small}
and
\begin{eqnarray*}
  v & = & 2
  +\frac{25}{64} s^2
  -\frac{125}{128}s^3
  +\frac{9375}{4096}s^4 \, , \\
  w & = & 2
  - \frac{1}{4}s
  + \frac{5}{16} s^2
  - \frac{25}{64} s^3
  +\frac{125}{256} s^4 \, .
\end{eqnarray*}
On the branch for $s < 0$, the eigenvalues of the diagonal matrix $X$ can be read off
from the diagonal elements and the eigenvalues of the matrix $W$ are both zeroes,
i.e., in that branch $W$ is the zero matrix. On that branch, the eigenvalues of $Y$
are $1$ and $0$ and the eigenvalues of $V$ are 
$\frac{(5s+4)\sqrt{2}}{4\sqrt{5s+2}}$ and $0$. With respect to a joint system of 
eigenvectors, we have a pairing of the eigenvalues of $Y$ and of $V$ such that
the product always gives zero.

At some value $t_2$, the next event point occurs. Here, the left upper entry
of $X$ becomes 0 and thus the right lower entry becomes 1. Using
a parametric description in $t$, we can symbolically compute
the $t$-value, where the left upper entry of $X$ becomes zero.
Namely, this happens at
\[
  t_2 \ = \ \frac{129}{125} - \frac{4 \sqrt{26}}{125} \ \approx \ 
  0.86883 .
\]
For values smaller than $t_2$, $X$ remains the matrix
\[
    X = \left( \begin{array}{cc}
    0 & 0 \\ 0 & 1
  \end{array} \right) \, .
\]
For $t$-values slightly smaller than 1, the matrix $Y$ is a rank 2 matrix
and the nondiagonal entries increase from zero and take nonzero-values.
In fact, locally one of the eigenvalues decreases starting from 1 and the
other 1 increases starting from 0.
At $t=0$, the next event point occurs and we arrive at
the Nash equilibrium
\[
    X \ = \ \left( \begin{array}{cc}
    0 & 0 \\ 0 & 1
  \end{array} \right) \, , \;
    Y \ = \
    \left( \begin{array}{cc}
    \frac{1}{2} + \frac{5\sqrt{26}}{52} & \frac{\sqrt{26}}{52} \\
    \frac{\sqrt{26}}{52} & \frac{1}{2} - \frac{5\sqrt{26}}{52}
  \end{array} \right)  \ \approx \
    \left( \begin{array}{cc}
    0.9903 & 0.0981 \\
    0.0981 & 0.0097
  \end{array} \right) .
\]
\end{ex}

Note that in our example, besides the event point $(X^{(1)}, Y^{(1)})$,
there exists another point which satisfies the 
are solutions of the equations in Lemma~\ref{le:equations1},
namely the point $X = \begin{pmatrix} 1 & 0 \\ 0 & 1 \end{pmatrix}$,
$Y = \begin{pmatrix} 0 & 1 \\ 0 & 1 \end{pmatrix}$ which is
a Nash equilibrium of the parametric game $G(t)$ for $t=1$.
However for this point, the matrix $V(X(t))$ is not positive semidefinite.

\section{On the smoothness of the curve\label{se:smooth}}

Clearly, the event points are non-smooth points on the 
Lemke-Howson path. 
In this section, we study the smoothness of the path segments
between two event points. We show that if a game is non-degenerate
then the path segments between two event points are smooth curves.
To this end, we study the derivatives of the 
equations. 

To reveal the interplay of the derivatives with
the combinatorics, it is useful to start from the polyhedral
situation of usual bimatrix games. Let $(A,B)$ be a bimatrix game,
where we can assume $A,B > 0$, that is, all entries are positive.
A point on the Lemke-Howson path is a tuple $(x,y,t)$ such that 
$$
f(x,y,t)=x\odot (\one_m - A_ty) = 0\in \R^m, \quad g(x,y,t)=(\one_n - B^Tx) \odot y = 0 \in \R^n,
$$
where $\odot$ denotes the Hadamard product, i.e., the componentwise multiplication.

In this setting, any non-event point is obviously smooth by construction, as the paths are all piecewise linear. 
Regardless of this fact, we include a formal proof for this statement in terms of the derivatives,
which we will generalize in the following 
to prove the smoothness in the case of semidefinite games.

By the Jacobian criterion, a point $(x,y,t)$ is smooth on the curve, if the Jacobian matrix of 
$\begin{pmatrix}
f \\
g
\end{pmatrix}$ with respect to $x$ and $y$ is regular, i.e., it
has full rank.
Here, the Jacobian is the $(m+n)\times (m+n)$ matrix
\[
  \left(
  \begin{array}{cc}
    \diag(\one_m - A_ty) & - \diag(x) \cdot A_t \\
    - \diag(y) B^T & \diag(\one_n - B^T x)
  \end{array}
  \right).
\]

At any point on the Lemke-Howson path, which is not an event point, the Jacobian is indeed regular.
However, a more useful viewpoint for us in this context is to use differentials.
Consider the differentials
\begin{eqnarray*}
d_{x,y} f(x,y,t) & = &
  dx \odot (\one_m - A_ty) - x \odot A \, dy \, , \\
d_{x,y} g(x,y,t) & = &
 - B^Tdx \odot y + (\one_n - B^T x)\odot dy \, .
\end{eqnarray*}

We record the following conversion between the view as differential
and the view as a Jacobian matrix, which we will prove in a more general version later on.
\begin{lemma}\label{le:Jac-differential0}
The Jacobian of 
$\begin{pmatrix}
f\\
g
\end{pmatrix}
=
\begin{pmatrix}
f(x,y,t) \\
g(x,y,t)
\end{pmatrix}
$
is regular if and only if there does not exist
$(dx,dy) \neq (0,0) \in \R^m \times \R^n$ such that both differentials $d_{x,y} f(x,y,t)$ and $d_{x,y} g(x,y,t)$ equal zero.
\end{lemma}

We can now verify that for bimatrix games every non-event point of the Lemke- Howson path is smooth.

\begin{proof}
Let $(x,y,t)$ be a non-event point of the Lemke-Howson path, that is $f(x,y,t)=0$, $g(x,y,t)=0$. Let $K:=\supp(x), \; K^c := [m]\backslash K, \; L:=\supp(y), \; L^c := [n]\backslash L$. By construction, $\supp(\one_m - A_ty)=K^c$ and $\supp(\one_n - B^T_tx)=L^c$.

Let $d_{x,y} f(x,y,t)=0$ and $d_{x,y} g(x,y,t)=0$. The property
$d_{x,y} f(x,y,t)=0$ implies $\supp(dx)\subset K^c$ and similarly 
we obtain $\supp(dy)\subset L^c$.
Since the game is non-degenerate, the columns of the $|K| \times |L^c|$ submatrix of $B^T$ with entries $b_{kl}$ for $k\in K, \, l\in L^c$ are linearly independent. Therefore $dx=0$ and analogously we see that $dy=0$.
This proves the claim that the point $(x,y,t)$ is smooth.
\end{proof}

For semidefinite games, it is beneficial to describe the derivatives
in terms of differentials. We start from
\begin{eqnarray*}
  F(X,Y,t) & = & X \cdot (I_m - 
	\Phi_A(Y)) \, , \\
  G(X,Y,t) & = & Y \cdot (I_n - 
  	 \Phi_B'(X)) \, .
\end{eqnarray*}
Note that $\Phi_A$ depends on $t$. However, in our considerations $t$ is a fixed parameter and we therefore omit the dependency from $t$ in the notation.
The differentials are

\begin{eqnarray*}
D_{X,Y} F(X,Y,t)[H,K] & = &
  (I_m - \Phi_A(Y)) H
  - X \cdot \Phi_A(K) \, , \\
D_{X,Y} G(X,Y,t)[H,K] & = &
 - Y \cdot \Phi_B'(H)
 + (I_n - \Phi_B'(X)) K \, .
\end{eqnarray*}

We record the following conversion between the view as differential
and the view as a Jacobian matrix similar to the Lemma~\ref{le:Jac-differential}.

\begin{lemma}\label{le:Jac-differential}
The Jacobian of $\begin{pmatrix}
F\\
G
\end{pmatrix}$ is regular if and only if there does not exist
$(H,K) \neq (0,0)\in \R^{m\times m} \times \R^{n\times n}$ such that 
\begin{eqnarray}
    \label{eq:charact-jacobian1}
& &  D_{X,Y} F(X,Y,t) [H,K]= 0 \in \R^{m\times m}\\
&\text{and}&  D_{X,Y} G(X,Y,t) [H,K]= 0 \in \R^{n\times n}.
\end{eqnarray}
\end{lemma}

\begin{proof}
The rows of the Jacobian $J$ are indexed by either $(i,j) \in [m] \times [m]$ 
(referring to $F$)
or by $(i,j) \in [n] \times [n]$ (referring to $G$). 
Similarly, a vector $v \in \R^{m^2 + n^2}$, can be written as a pair of
two matrices, say, called $H^* \in \R^{m \times m}$ and $K^* \in \R^{n \times n}$.

For the rows of $J$ that refer to $F$, the matrix-vector product of the row
indexed by $(i,j)$ and $v$ evaluates to
\[
  \sum_{k,l=1}^m \left( \frac{\partial}{\partial X_{kl}} F_{ij} \right) H^*_{kl}
\]
and analogously for the rows referring to $G$
\[
  \sum_{k,l=1}^n \left( \frac{\partial}{\partial Y_{kl}} G_{ij} \right) K^*_{kl} \, .
\]
Hence, the vector $v$ is contained
in the kernel of $J$ if and only if the matrix~\eqref{eq:charact-jacobian1}
evaluates to the zero matrix at $(H^*,K^*)$. Since a matrix is 
singular if and only if it contains a non-zero vector in the kernel, the claim follows.
\end{proof}
In the diagonal situation, this coincides with the previous discussion of the bimatrix case, where the matrices $X$ and $Y$ are diagonal matrices. 
Thus, when we consider the differentials, only the elements on the diagonal matter matter, that is, $dX_{ii}$ and $dY_{jj}$.

Note that when considering the derivatives or 
differentials for points on the central path,
we can drop the symmetry requirement for our matrices. This is well known
in similar contexts, for example, the central path in semidefinite
programming (\cite[proof of Theorem 3.3]{deklerk-2003}). Namely, we know
that in the relative interior of the curve branches between two event points,
there exists a solution which meets the symmetry requirements. 
If we show that the Jacobian
in the space without symmetry requirements is non-degenerate then
this implies that there cannot be any other infinitesimally near
unsymmetric solution on the central path.

We arrive at the following sufficient condition for the
smoothness of the Lemke-Howson path in a given point.

\begin{thm}
Consider a non-degenerate semidefinite game and let $(X,Y,t)$ be on the Lemke-Howson path and 
neither an event point nor a Nash equilibrium. 
Then $(X,Y,t)$ is a smooth point on the Lemke-Howson path.
\end{thm}
\begin{proof}
Let $X \in \mathcal{X}$ and $Y \in \mathcal{Y}$ be  of rank $k$ and $t>0$ be fixed. Recall the notation $W(Y)= I_m -\Phi_{A}(Y)$ and $V(X)= I_n - \Phi_{B}(X)$.

By construction, $\im (X)=\ker(W(Y)), \; \ker (X)=\im(W(Y)), \; \im (Y)=\ker(V(X))$ and $\ker (Y)=\im(V(X))$, where $\im$ and $\ker$ denote the image and the kernel of a matrix.
Now recall the functions
\begin{eqnarray*}
  F(X,Y,t) \ = \ X \cdot W(Y) \, , \quad
  G(X,Y,t) \ = \ Y \cdot V(X) \, .
\end{eqnarray*}
By the theorem of the implicit function and lemma \ref{le:Jac-differential}, $(X,Y)$ is a smooth point on the Lemke-Howson path 
if the differentials $D_{X,Y} F [H,K]$ and $D_{X,Y} G [H,K]$ do not vanish for any $(H,K)\neq (0,0)$. 

First we observe what it means for $D_{X,Y} F [H,K]$ to vanish,
\begin{eqnarray*}
 D_{X,Y}F [H,K] = W(Y)H - X\Phi_A(K)=0 .
\end{eqnarray*}
That implies $W(Y)H = 0 = X\Phi_A(K)$ because (every column of) $W(Y)H$ is contained in the kernel of $X$ and (every column of) $X\Phi_A(K)$ is contained in the image of $X$. Similarly, if
\begin{eqnarray*}
D_{X,Y}G [H,K] = - Y\Phi_B'(H) + V(X)K=0
\end{eqnarray*}
then $V(X)K = 0 = Y\Phi_B'(H)$.
We are left to show that $(H,K)=(0,0)$.
Consider $X':=X + \lambda H$ for some arbitrary and fixed $\lambda \in \R$.
We can see that
\begin{eqnarray*}
X' W(Y)= X W(Y) + \lambda H W(Y) & =&  0, \\
Y V(X') = Y (I_n - \Phi_B'(X + \lambda H))= Y (I_n - \Phi_B'(X) - \lambda \Phi_B'(H))= YV(X) & =& 0.
\end{eqnarray*}
Since the game is non-degenerate, only finitely many matrices $(X',Y)$ can be in both of those varieties, therefore $H=0$.
Using an analogous argument, we see $K=0$, which completes the proof.
\end{proof}

\section{Conclusion and outlook}

We have provided a generalization of Lemke–Howson paths for bimatrix games to semidefinite games, offering both a structural perspective on the Nash equilibria of semidefinite games and numerically-algebraic algorithmic techniques for computing them.
Our discussion of Lemke–Howson paths began with a formulation of a semidefinite game as a semidefinite complementarity problem.
Analogous to the connection between Lemke–Howson paths and linear complementarity problems in the bimatrix case, our approach for semidefinite games naturally extends to the broader class of semidefinite complementarity problems, which encompass a range of noncommutative complementarity problems.

The investigation of Lemke–Howson paths has also led to natural questions regarding degeneracy in semidefinite games. In the case of bimatrix games, degeneracy is well understood, and several equivalent formulations exist from different perspectives (see \cite{von-stengel-handbook}).
We have initiated the study of degeneracy in semidefinite games, and a key direction for future research is to understand how the various notions of degeneracy known for bimatrix games can be extended and adapted to the semidefinite setting.

\subsubsection*{Acknowledgments}

CI, TT and ET are partially supported by the joint PROCOPE project ``Quantum games
and polynomial optimization'' of the MEAE/MESR and the DAAD (57753345).
TT is partially supported by the DFG Priority Program
``Combinatorial Synergies'' (grant no.\ 539847176).
ET is partially supported by the PGMO grant SOAP (Sparsity in Optimization via Algebra and Polynomials).
AV is partially supported by the National Research Foundation, Singapore under its QEP2.0 programme (NRF2021-QEP2-02-P05).
\bibliographystyle{abbrv}
\bibliography{games}
	
\end{document}